\documentclass [12pt]{amsart}
\usepackage{amssymb,latexsym,amsfonts, amsthm, amsmath, amsrefs, mathtools,color}
\usepackage{array}
\usepackage{enumerate}
\usepackage{booktabs}
\usepackage{fancyhdr}
\usepackage{verbatim}
\usepackage[margin=1in,twoside=false]{geometry}
\usepackage{multicol}
\usepackage{longtable}
\usepackage{hyperref}
\usepackage[normalem]{ulem}
\usepackage[capitalize]{cleveref}
\usepackage{mathrsfs}
\usepackage{tikz}
\usetikzlibrary{matrix,arrows}

\newcommand{\lra}{\leftrightarrow}
\newcommand{\aff}{\mathrm{aff}}
\newcommand{\pcode}{\mathbf{P}_{\mathbf{code}}}

\crefname{figure}{Figure}{Figures}
\crefname{prop}{Proposition}{Propositions}

\theoremstyle{definition}

\newtheorem{thm}{Theorem}[section]   
\newtheorem{cor}[thm]{Corollary}     
\newtheorem{lem}[thm]{Lemma}         
\newtheorem{prop}[thm]{Proposition}  
\newtheorem{const}[thm]{Construction}  

\theoremstyle{definition} 
\newtheorem{defn}[thm]{Definition}   
\newtheorem*{conj*}{Conjecture}        
\newtheorem{ex}[thm]{Example}        

\newtheorem{question}[thm]{Question}

\DeclareMathOperator{\code}{code}

\DeclareMathOperator{\tk}{Tk}

\DeclareMathOperator{\odim}{odim}
\DeclareMathOperator{\cdim}{cdim}

\DeclareMathOperator{\conv}{conv}

\def\C{\mathbb C}

\def\R{\mathbb R}

\def\cC{\mathcal C}
\def\cD{\mathcal D}

\def\cL{\mathcal L}

\def\cP{\mathcal P}

\def\cR{\mathcal R}

\def\cT{\mathcal T}
\def\cU{\mathcal U}

\def\od{:=}

\renewcommand{\emptyset}{\varnothing}

\def\1{\overline{1}}
\def\2{\overline{2}}
\def\3{\overline{3}}
\def\4{\overline{4}}
\def\5{\overline{5}}

\title{Order-forcing in Neural Codes}
\author{R. Amzi Jeffs, Caitlin Lienkaemper, and Nora Youngs}

\begin{document}
\maketitle
\begin{abstract}
   Convex neural codes are subsets of the Boolean lattice that record the intersection patterns of convex sets in Euclidean space. Much work in recent years has focused on finding combinatorial criteria on codes that can be used to classify whether or not a code is convex. In this paper we introduce \emph{order-forcing}, a combinatorial tool which recognizes when certain regions in a realization of a code must appear along a line segment between other regions. We use order-forcing to construct novel examples of non-convex codes, and to expand existing families of examples. We also construct a family of codes which shows that a dimension bound of Cruz, Giusti, Itskov, and Kronholm (referred to as monotonicity of open convexity) is tight in all dimensions.
\end{abstract}
\section{Introduction}
A \emph{combinatorial neural code} or simply \emph{neural code} is a subset of the Boolean lattice $2^{[n]}$, where $[n]\od \{1,2,\ldots,n\}$. 
A neural code is called \emph{convex} if it records the intersection pattern of a collection of convex sets in $\R^d$. More specifically, a code $\cC\subseteq 2^{[n]}$ is \emph{open convex} if there exists a collection of convex open sets $\cU = \{U_1, \ldots, U_n\}$ in $\R^d$ such that 
\[\sigma\in \cC \, \Leftrightarrow \, \bigcap_{i\in \sigma} U_i \setminus \bigcup_{j\notin \sigma}U_j \neq \emptyset.\]
The region $\bigcap_{i\in\sigma} U_i\setminus \bigcup_{j\notin\sigma} U_j$ is called the \emph{atom} of $\sigma$ in $\cU$, and is denoted $A^\cU_\sigma$. The collection $\cU$ is called an \emph{open realization} of $\cC$, and the smallest dimension $d$ in which one can find an open realization is called the \emph{open embedding dimension} of $\cC$, denoted $\odim(\cC)$. Analogously, one may consider closed convex sets to obtain notions of \emph{closed convex} codes, \emph{closed realizations}, and \emph{closed embedding dimension}.

The study of convex codes is motivated by the problem in neuroscience of characterizing which patterns of neural activity can arise from neurons with approximately convex receptive fields, such as hippocampal place cells \cites{okeefe, reconstruction}. Further, the task of characterizing convex codes is  mathematically interesting in its own right. Explaining why certain codes are not convex has required the development of new and interesting theorems in discrete geometry \cites{sunflowers, embeddingphenomena}, and has connections to established mathematical theories, such as that of oriented matroids \cite{matroids}. There is a sizeable body of work on determining combinatorial criteria that can detect whether or not a code is convex (in addition to above cited works, see \cites{nogo, neuralring13, obstructions, openclosed, goldrup2020classification, morphisms, nomonotone, undecidability, sparse, CUR, chan2020nondegenerate}), but a complete characterization of convexity remains elusive.

In this paper, we introduce a combinatorial concept that we call \emph{order-forcing} (see Definition~\ref{def:orderforced}). Order-forcing provides an elementary connection between the combinatorics of a code and the geometric arrangement of atoms in its open or closed realizations, as described in Theorem \ref{thm:order-forcing} below.  
\begin{thm}\label{thm:order-forcing}
Let $\sigma_1, \sigma_2, \ldots, \sigma_k$ be an order-forced sequence of codewords in a code $\cC\subseteq 2^{[n]}$. Let $\cU = \{U_1, \ldots ,U_n\}$ be a (closed or open) convex realization of $\cC$, and let $x\in A^{\cU}_{\sigma_1}$, and $y\in A^{\cU}_{\sigma_k}$. Then the line segment $\overline{xy}$ must pass through the atoms of $\sigma_1, \sigma_2, \ldots, \sigma_k$, in this order.
\end{thm} 
We prove Theorem \ref{thm:order-forcing} in Section \ref{sec:order-forcing}, and also provide several basic examples of order-forcing. 

 Another important class of codes is those which can be realized using a good covers. A \emph{good cover} is a collection of sets $\{U_1,\ldots, U_n\}$ in $\R^d$ (all open or all closed) such that $\bigcap_{i\in\sigma} U_i$ is either empty or contractible for all nonempty $\sigma\subseteq [n]$. Every convex realization is also a good cover, but not every good cover code is convex. Previous examples of non-convex good cover codes have required technical geometric results to describe (see \cites{obstructions, sunflowers, undecidability, CUR}). In Section~\ref{sec:of_examples} we use order-forcing to describe new good cover codes that are not convex:
\begin{itemize}
    \item We generalize a minimally non-convex code from \cite{morphisms} based on sunflowers of convex open sets to a family $\{\cL_n\mid n\ge 0\}$ of minimally non-convex  good cover codes (Proposition \ref{prop:stretch_sun}).
    \item We build a good cover code $\cR$ that is neither open nor closed convex by using order-forcing to guarantee that two disjoint sets would cross one another in a convex realization of $\cR$ (Proposition \ref{prop:rowboat}). This example is notable in that it relies on the order that codewords appear along line segments, rather than just certain codewords being ``between" one another.
    \item We build a good cover code $\cT$ that is neither open nor closed convex by using order-forcing to guarantee a non-convex ``twisting" in every realization of $\cT$ (Proposition \ref{prop:braid}).

\end{itemize}

These examples illustrate the utility of order-forcing. The codes $\cT$ and $\cR$ have the advantage that they require only elementary geometric techniques (i.e. order-forcing) to analyze. The codes $\cT$ and $\cR$ are also the first ``natural" examples we know of of good cover codes which are neither open nor closed convex (in the sense that they are not the disjoint union of a code which is open but not closed convex and a code which is closed but open convex).

In Section 4 we build on results of \cite{openclosed}, using tools from \cite{sunflowers} to prove that adding a non-maximal codeword to a code may increase its open embedding dimension, no matter what value the open embedding dimension has. More precisely, we describe for every $d\ge 1$ a code $\cP_d$ such that $\odim(\cP_d) = d$ and $\odim(\cP_d\cup \{\sigma\}) = d+1$ for some new non-maximal codeword $\sigma$. This shows that a monotonicity bound described in \cite{openclosed} is tight for the family of codes $\{\cP_d\mid d\ge 1\}$. The techniques used in Section \ref{sec:strict_mono} hint at potential generalizations of order-forcing, which we discuss in Section \ref{sec:conclusion} along with other open questions.

\section{Order-Forcing \label{sec:order-forcing}}

When we constrain ourselves to realizations that use only open (or only closed) convex regions $\{U_i\}$, we restrict not only which codes may be realized, but how regions in these realizations can be arranged. 
In particular, when we move along continuous paths through realizations composed of open (or closed) sets $U_i$, we are limited in the transitions we can make from one atom to the next.  

\begin{lem}\label{lem:edges} Suppose $\cC\subseteq 2^{[n]}$ is a neural code with a good cover realization $\cU$, and let $\sigma$ and $\tau$ be codewords of $\cC$. If there are points $p_\sigma\in A^\cU_\sigma$ and $p_\tau \in A^\cU_\tau$ and a continuous path from $p_\sigma$ to $p_\tau$ that is contained in $A^\cU_\sigma\cup A^\cU_\tau$ (that is, if the atoms are adjacent in the realization), then either $\sigma\subseteq\tau$ or $\tau\subseteq\sigma$.
\end{lem}

\begin{proof} Let $P$ be the image of a continuous path from $p_\sigma$ to $p_\tau$ with $P\subseteq A^\cU_\sigma\cup A^\cU_\tau$. Suppose for contradiction that $\sigma\not\subset\tau$ and $\tau\not\subset \sigma$. Then there exist elements $i\in \sigma\setminus \tau$ and $j\in \tau\setminus\sigma$. But then $U_i\cap P$ and $U_j\cap P$ partition $P$ (every point in $P$ is in exactly one of $A_\sigma^\cU$ or $A_\tau^\cU$ and thus in exactly one of $U_i$ or $U_j$). Since our good cover consists of sets that are all open or all closed, the sets $U_i\cap P$ and $U_j\cap P$ are both relatively open or both relatively closed in $P$. This contradicts the fact that $P$ is connected, so $\sigma\subseteq\tau$ or $\tau\subseteq \sigma$ as desired.
\end{proof}

Thus, as we move continuously through any good cover realization of a code, we are moving along  edges in the following graph $G_\cC$:

\begin{defn}\label{def:codewordgraph}
Let $\cC\subseteq 2^{[n]}$ be a neural code. The \emph{codeword containment graph} of $\cC$ is the graph $G_{\cC}$ whose vertices are codewords of $\cC$, with edges $\{\sigma,\tau\}$ when either $\sigma\subsetneq\tau$ or $\tau\subsetneq \sigma$. Note that this graph is also defined in \cite{chan2020nondegenerate}. 
\end{defn}

\begin{ex}\label{ex:manypaths}
Consider the code $\cC = \{{\bf 1235}, {\bf 1245}, {\bf 1256}, 125, 13, 14, 15, \emptyset\}$. The graph $G_\cC$ for this code is shown in Figure \ref{fig:manypaths}.

\begin{center}
\begin{figure}[h!]
\includegraphics[width=2in]{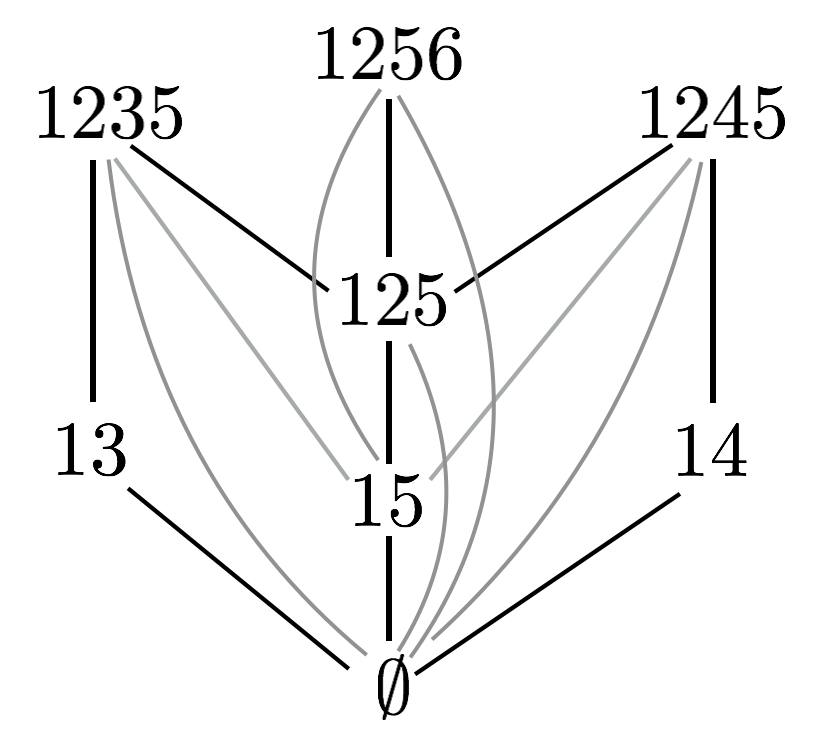}
\caption[something]{The codeword containment graph $G_\cC$ for the code in Example \ref{ex:manypaths}. \label{fig:manypaths}}
\end{figure} 
\end{center}
\end{ex}

Lemma \ref{lem:edges} implies that any continuous path from one codeword region to another codeword region in an open (or, closed) realization of the code $\cC$ must correspond to a walk in the graph $G_\cC$. For straight-line paths within a convex realization, this walk must respect convexity, a property we call {\it feasibility} (see Lemma \ref{lem:convexfeasible}). 
    
\begin{defn}\label{def:feasible}
Let $\cC$ be a neural code and $G_\cC$ its codeword containment graph. A $\sigma,\tau$ walk  $\sigma=v_1,v_2,...,v_k=\tau$ in $G_\cC$ is called \emph{feasible} if $v_i\cap v_j\subseteq v_m$ for all  $1\le i<m<j\le k$.
\end{defn}

In general, if there exists a feasible $\sigma,\tau$ walk, then by removing portions of the walk between repeated vertices, we can obtain a feasible $\sigma, \tau$ path.  This does not, however, mean that there is a corresponding straight line path in the realization which would follow precisely this sequence of codewords. For example, one could form a closed realization of the code $\{\mathbf{1},\mathbf{2},\mathbf{3},\emptyset\}$ in which $U_2$ is a hyperplane, and $U_1$ and $U_3$ are contained in its positive and negative side respectively. Then any straight line from the atom of $1$ to the atom of $3$ must pass through the atom of 2, but the path $1, \emptyset, 3$ is a feasible path in $G_{\C}$ regardless.

\begin{ex}[Example \ref{ex:manypaths} continued] Consider the codewords $\sigma = 13$ and $\tau = 14$ in $\cC$ from the code in Example \ref{ex:manypaths}. There are many $\sigma,\tau$ walks; however, not all are feasible. For example, the walk $13, 1235, 15, 1245, 14$ would not be feasible; however, the walk $13, 1235, 125, 1275, 125, 1245$ is a feasible $\sigma,\tau$ walk. This walk contains the feasible path $13, 1235, 125, 1245, 14$, which in this case is the unique  feasible $\sigma,\tau$ path. 

\end{ex}

\begin{lem}\label{lem:convexfeasible} Suppose $\cC\subseteq 2^{[n]}$ is a neural code with a convex realization $\cU = \{U_1,\ldots, U_n\}$, and let $\sigma$ and $\tau$ be codewords of $\cC$. If there are points $p_\sigma\in A^\cU_\sigma$ and $p_\tau \in A^\cU_\tau$, then the sequence of atoms along the line $\overline{p_\sigma p_\tau}$ forms a feasible  walk in $G_\cC$.
\end{lem}

\begin{proof}   Select points $p_\sigma\in A^\cU_\sigma$ and $p_\tau\in A^\cU_\tau$, and let the sequence of atoms along the line be given by $\sigma=\tau_1,\tau_2,...,\tau_\ell = \tau$. 
By Lemma \ref{lem:edges}, if we cross directly from $A_{\tau_i}^\cU$ to $A_{\tau_{i+1}}^\cU$ along the path $\overline{xy}$, then either $\tau_i\subseteq \tau_{i+1}$ or $\tau_{i+1}\subseteq \tau_i$, thus $(\tau_i, \tau_{i+1})$ is an edge of $G_\cC$. 
Thus, $\tau_1, \ldots, \tau_\ell$ is a walk in $G_\cC$. 
To check feasibility, we need to show that for all $i \leq j \leq k$,  $\tau_i \cap\tau_k  \subseteq \tau_j$. 
We can choose points $x_i, x_j$, and $x_j$ in this order along $\overline{p_\sigma p_\tau}$ such that $x_i \in A_{\tau_i}^\cU$, $x_j \in A_{\tau_j}^\cU$, $x_k \in A_{\tau_k}^\cU$. 
Since $\{U_1, \ldots, U_n\}$ is a convex realization and intersections of convex sets are convex, $U_{\tau_i\cap \tau_k}$ is a convex set.
By the definition of a convex set, the line segment $\overline{x_i x_k}$ is contained in $U_{\tau_i\cap \tau_k}$. 
Thus, $x_j\in U_{\tau_i\cap \tau_k}$, so $\tau_i \cap \tau_k \subseteq \tau_j$. Thus, $\tau_1, \ldots, \tau_\ell$ is a feasible walk. 

\end{proof}

The idea of feasibility gives us a new tool for finding possible obstructions to convexity. In any convex realization of a code, straight line paths between points in the same set $U_i$ must correspond to feasible walks in the graph, and so codes where feasible walks are rare or nonexistent can force us into contradictions. To that end, we define a few particular restrictions we will encounter.

\begin{defn}\label{def:forcedbetween} Let $\cC$ be a neural code and $G_\cC$ its codeword containment graph. We say a vertex $v$ of $G_\cC$ is \emph{forced between} vertices $\sigma$ and $\tau$ if every feasible $\sigma,\tau$ path passes through $v$. 
\end{defn}

\begin{ex}[Example \ref{ex:manypaths} continued]
In the codeword containment graph $G_\cC$, we see that $1245$ is forced between $14$ and $15$. There are many possible feasible paths from $14$ to $15$ (for example (14, 1245, 15) or (14, 1245, 125, 15) or (14, 1245,125, 1235, 15), but all these paths must use $1245$. 
\end{ex}

In cases where there are multiple codewords forced between two vertices of our graph, we often find that these vertices are also forced into a particular order, a situation we call {\it order-forcing}.

\begin{defn} Let $\cC$ be a neural code and $G_\cC$ its corresponding graph. A sequence of codewords $\sigma_1,...,\sigma_k$ is \emph{order-forced} if every feasible $\sigma_1,\sigma_k$ path contains these codewords as a subsequence. 
\end{defn}





\begin{defn}\label{def:orderforced} Let $\cC$ be a neural code and  $G_\cC$ its corresponding graph.  A feasible path $\sigma_1,...,\sigma_k$ is \emph{strongly order-forced} if $\sigma_1,...,\sigma_k$ is the unique feasible $\sigma_1,\sigma_k$ walk in  $G_\cC$. 
\end{defn} 



\begin{ex}
Consider the code $$\cC = \{{\bf 2456},{\bf 123},{\bf 145},{\bf 437}, {\bf467},45,46,47,1,2,3,\emptyset\}$$

In this code, the sequence $145,45,2456,46,467,47,473$ is strongly order-forced.  In order to have a path in $G_\cC$ from $145$ to $437$ which is feasible, we can certainly only use codewords which contain $4$. If we restrict to the portion of  $G_\cC$ which contains $4$, we have that this subgraph is a path with endpoints $145$ and $437$. Thus, there is a unique path from $145$ to $437$, and we can check that this path is feasible. 

\end{ex}



\begin{proof}[Proof of Theorem \ref{thm:order-forcing}]
Let $\sigma_1, \ldots, \sigma_k$ be an order-forced sequence in a code $\cC.$ Let $\cU = \{U_1, \ldots, U_n\}$ be a (closed or open) convex realization of $\cC$. 
Let $x\in A_{\sigma_1}^{\cU}$ and $y\in A_{\sigma_k}^\cU,$ and let $\overline{xy}$ be the line segment from $x$ to $y$. Let $\tau_1 = \sigma_1, \tau_2, \ldots, \tau_\ell = \sigma_k$ be the sequence of atoms along $\overline{xy}$.
By Lemma \ref{lem:convexfeasible}, we have that $\tau_1, \ldots, \tau_\ell$ is a feasible  walk from $\sigma_1$ to $\sigma_k$ in $G_\cC.$
Since every feasible walk from $\sigma_1$ to $\sigma_k$ contains a feasible path from $\sigma_1$ to $\sigma_k$, and every feasible path from $\sigma_1$ to $\sigma_k$ contains  $\sigma_1, \sigma_2, \ldots, \sigma_k$ as a subsequence, this suffices to prove Theorem \ref{thm:order-forcing}.

\end{proof}

 The situation where a codeword $v$ is forced between $\sigma$ and $\tau$ is a special case of order-forcing, and in this case we obtain the following result.  Once we know that a sequence is order-forced in a code $\cC$, we are often able to obtain several instances of order-forcing.
\begin{cor}\label{cor:forcedbetween}  
 Let $\cC\subseteq 2^{[n]}$ and suppose $\cU = \{U_1, \ldots ,U_n\}$ is a (closed or open) convex realization of a code $\cC$. If  $v\in \cC$ is forced between $\sigma$ and $\tau$, then for any $x\in A^{\cU}_{\sigma}$, and $y\in A^{\cU}_{\tau}$, the line segment $\overline{xy}$ must pass through the atom of $v$. 
\end{cor}



In the following example, we illustrate the value of these ideas by showing a proof that a relatively small code is open-convex but not closed-convex.

 \begin{ex}\label{ex:wheel} In \cite{openclosed}, the authors provide an example of a code which is open convex, but not closed convex. This example (in particular the proof that it has no closed convex realization) is an instance of order-forcing, though it was not described by that name in \cite{openclosed}. A slightly smaller example of a similar code which is closed convex, but not open convex appears as code C15 in \cite{goldrup2020classification}. In this example, we give a proof of this result which resembles the proof in \cite[Lemma 2.9]{openclosed}, but is written to make the use of order-forcing explicit. 

 The code $$\cC=\{{\bf 123},{\bf 234},{\bf 345},{\bf 145},{\bf 125},12,23,34,45,15,\emptyset\}$$
has an open convex realization, but does not have a closed convex realization.\\

To see that an open convex realization exists, we exhibit one in Figure \ref{fig:5circle}. 
\begin{center}
\begin{figure}[ht!]
\includegraphics[width=1.5in]{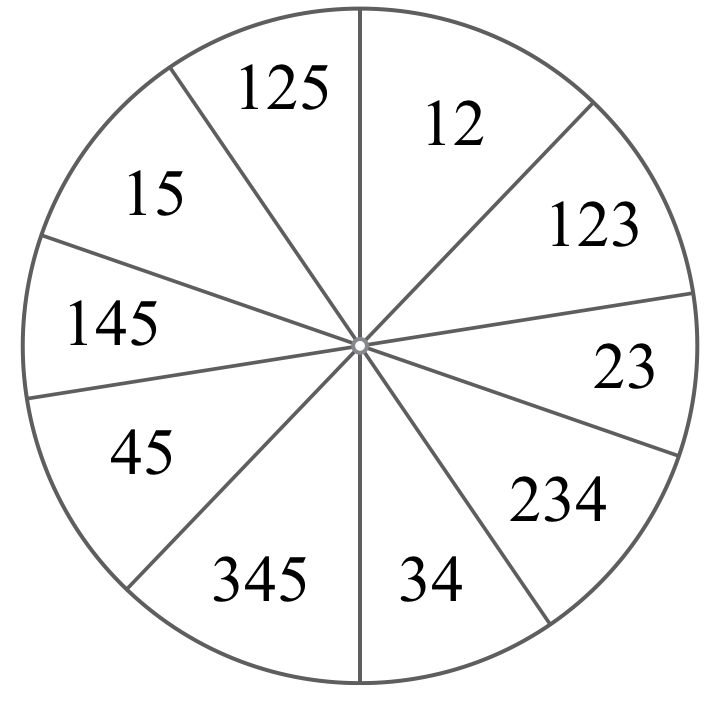}
\caption{A realization of $\cC$ using open sets. The regions $U_1,...,U_5$ are open half discs. }\label{fig:5circle}
\end{figure} 
\end{center}

To show that no closed convex realization may exist, we proceed by contradiction.  Suppose that for some $d\geq 1$, there exists a closed convex realization $\{U_1,\ldots, U_5\}$ of $\cC$ in $\R^d$.  Select points $p_{125}\in U_{125}$ and $p_{345}\in U_{345}$. 
Since both points are within the convex set $U_5$, the line segment $L_1$ from $p_{125}$ to $p_{345}$ is contained within $U_5$. Thus, it cannot pass through $U_{123}$. Note that $U_{123}$ is a closed set, as $123$ is a maximal codeword. Pick a point $p_{123}\in U_{123}$ which minimizes the distance to the set $L_1$; this is possible as these sets are disjoint and $L_1$ is compact. 

Now, consider the line segment $L_2$ from $p_{125}$ to $p_{123}$; note that $L_2\subset U_{12}$. In this code, $12$ is forced between $125$ and $123$, so by Corollary \ref{cor:forcedbetween} there exists a point $p_{12}$, between $p_{125}$ and $p_{123}$ along this line, which is in $A^{\mathcal U}_{12}$. Likewise, if we consider the line segment $L_3$ from $p_{123}$ to $p_{345}$, the order-forced sequence $123,23,234,34,345$ implies there is a point $p_{234}\in U_{234}$ on $L_3$ which is between $p_{123}$ and $p_{345}$. 

Finally, consider the line segment $L_4$ between $p_{12}$ and $p_{234}$. $L_4$ must pass through $U_{123}$ somewhere between these points because $123$ is forced between $12$ and ${234}$. Select a point $q_{123}$ on this line and within $U_{123}$; then, $q_{123}$ will be closer to $L_1$ than $p_{123}$, a contradiction. 
\end{ex}

\section{New Examples of Non-Convex Codes \label{sec:of_examples}}
In this section, we demonstrate the power of order-forcing by using order-forcing to construct a  new infinite family of minimally non-convex codes and two new non-convex codes. 

\subsection{Stretching sunflowers}\label{sec:stretch}

Early examples of good cover codes which are not convex come from a result about sunflowers of convex open sets,  described further in Section \ref{sec:strict_mono}.  The $d = 2$ case of this theorem was used as a lemma to give the first example of a non-convex good cover code in \cite[Theorem 3.1]{obstructions}. 
In this section, we give a new infinite family of non-convex codes generalizing this code.
In order to produce further examples of non-convex codes, we need a notion of what it means for a new code to be genuinely different from an old one. 
For instance, it is easy to produce ``new" non-convex codes by relabeling neurons, or by adding more neurons in some trivial way. 

A notion of minors for codes gives a partial answer to this problem. Minors were introduced in \cite{morphisms} via a partial order on all (appropriate equivalence classes of) codes, though were only called minors beginning in \cite{embeddingphenomena}. This framework has been further developed in \cites{sunflowers,matroids}. We define minors and related notions below.  

\begin{defn}\label{def:trunk}\label{def:morphism}
Let $\cC\subseteq 2^{[n]}$ be a code. A \emph{trunk} in $\cC$ is a set of the form \[
\tk_{\cC}(\sigma) \od \{c\in\cC \mid \sigma\subseteq c\},
\]
or the empty set. A \emph{morphism} is a function between codes such that the preimage of a trunk is again a trunk.
\end{defn}
\begin{defn}
A code $\cC$ is a \emph{minor} of $\cD$, written $\cC \leq \cD$, if there is a surjective morphism $f:\cD\to \cC$.
\end{defn}

The poset of all codes partially ordered via minors is denoted $\pcode$. The main relevance of minors in our context is the following. (See \cite[Section 9]{embeddingphenomena} for details.)

\begin{prop}
Suppose that $\cD$ is a minor of $\cC$. Then $\odim(\cD)\le \odim(\cC)$, and $\cdim(\cD)\le \cdim(\cC)$. 
\end{prop}

As a result, open (respectively closed) convexity is preserved when replacing a code by a minor. We may therefore define a notion of minimal non-convexity for open or closed convex codes. We say that a code $\cC$ is \emph{minimally non-convex} when it is not open convex, but all of its proper minors are open convex. 
For instance, \cite[Theorem 5.10]{morphisms} gives a minimally non-convex good cover code $$\cC_0 = \{{\bf 3456}, {\bf 123}, {\bf 145}, {\bf 256}, 45, 56, 1, 2, 3, \emptyset\}.$$  This code is a proper minor of the first non-convex good cover code, which is described in \cite[Theorem 3.1]{obstructions}.

In this subsection, we introduce a family of codes $\{\cL_n\mid n\ge 0\} $ which generalize $\cC_0$ to an infinite family of minimally non-convex codes. Geometrically, each of these codes is only a small modification of the code $\cC_0$, and has the same basic obstruction to convexity. This lies in contrast to  \cite[Theorem 4.2]{sunflowers}, which generalizes the non-convex code $\cC_0$ in  \cite[Theorem 5.10]{morphisms} to an infinite family of minimally non-convex codes by using higher-dimensional versions of the sunflower theorem.  Thus the family $\{\cL_n\mid n\ge 0\}$ demonstrates that the intuition that each minimally non-convex code should result from a ``new" obstruction to convexity does not hold.

\begin{center}
\begin{figure}[ht!]
\includegraphics[width = 5 in]{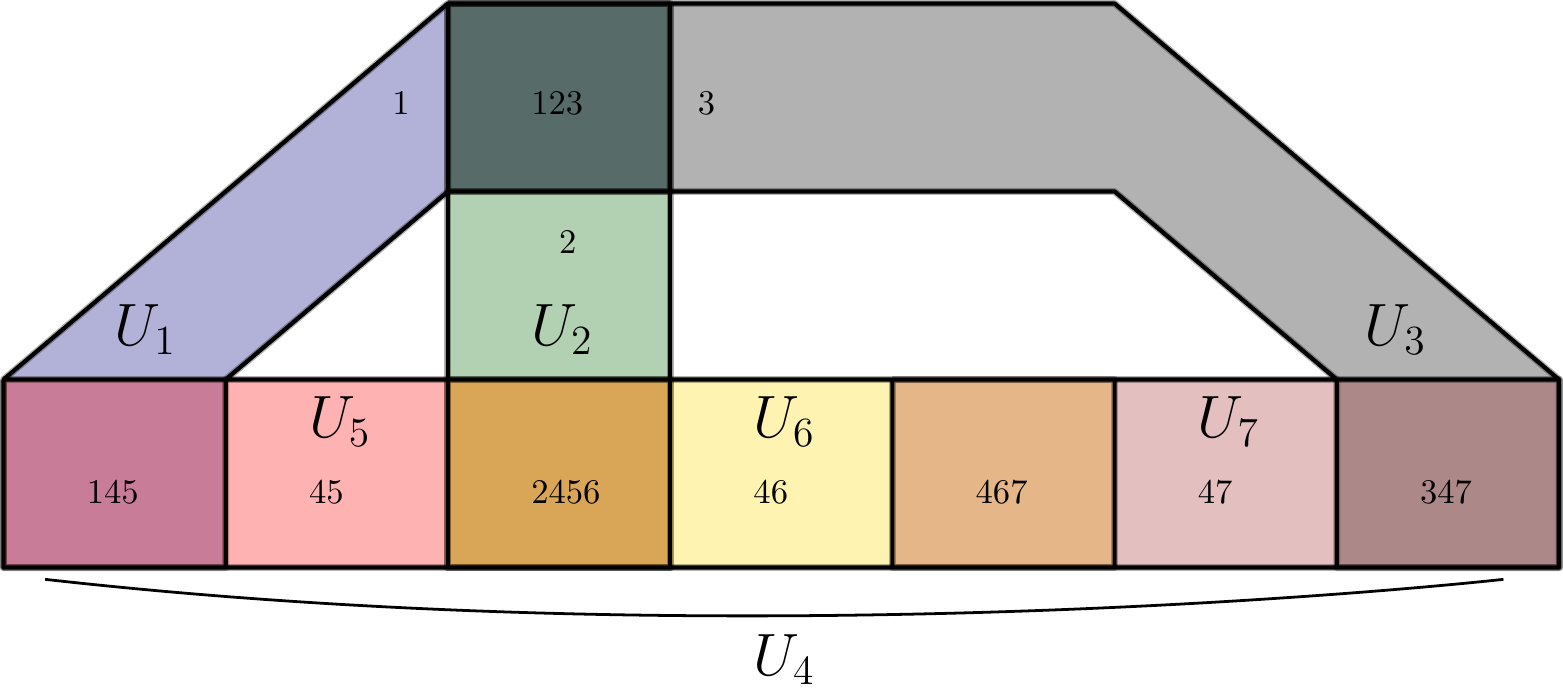}
\caption{$\cL_1 = \{\mathbf{2456}, \mathbf{123}, \mathbf{145}, \mathbf{437}, \mathbf{467}, 45, 46, 47, 1, 2, 3, \emptyset\}$}
\label{fig:L1}
\end{figure} 
\end{center}

\begin{defn}For $n\geq 0$, define the code 
$$\cL_{n} = \{\emptyset, 1, 2, 3, \mathbf{123}, \mathbf{145}, 45, \mathbf{2456}, 46, \mathbf{467}, 47, \mathbf{478}, \ldots, 4(n+6),34(n+6)\}$$. 
\end{defn}


For instance, $\cL_1 = \{\mathbf{2456}, \mathbf{123}, \mathbf{145}, \mathbf{437}, \mathbf{467}, 45, 46, 47, 1, 2, 3, \emptyset\}$. A good cover realization of $\cL_1$ is given in Figure \ref{fig:L1}.

Notice below that $\cL_0$ is equal to $\cC_0$ under the permutation of the neurons $2 \lra 3$ and $4\lra 5$. 
Thus, the family $\cL_n$ generalizes $\cC_0$. 
Even though each $\cL_n$ is minimally non-convex, the non-convexity of $\cL_0$ directly implies the non-convexity of each $\cL_n$ for $n > 0$.

\begin{prop}\label{prop:stretch_sun}
For $n\ge 0$, the code $\cL_n$ is a good cover code, but is minimally non-convex.
\end{prop}

\begin{proof}
We first show that $\cL_n$ is non-convex by induction on $n$. 
The base case, that $\cL_0$ is non-convex, is proven by \cite[Theorem 5.10]{morphisms} since $\cL_0$ is permutation equivalent to $\cC_0$. Now, we show that if $\cL_{n-1}$ is not convex, then neither is $\cL_n$.
We do this by proving the contrapositive: in any convex realization of $\cL_n$, we can merge the sets  $U_{n+5}$ and $U_{n+6}$ in a convex realization of $\cL_n$ to produce a convex realization of $\cL_{n-1}$.
That is, if $\{U_1, \ldots, U_{n+5}, U_{n+6}\}$ is a convex realization of $\cL_n$, then $\{V_1,\ldots, V_{n+5}\}$ is a convex realization of $\cL_{n-1}$ where $V_1 = U_1, \ldots, V_{n+4} = U_{n+4}, V_{n+5} = U_{n+5}\cup U_{n+6}$. 

This gives us two things to check. First, we must check that $\code(\{V_1, \ldots, V_{n+5}\}) = \cL_{n-1}$. If $\sigma$ is a codeword of $\cL_n = \code(\{U_1, \ldots, U_{n+6}\})$ which does not contain the neuron $n+6$, then $\sigma$ is still a codeword of $\code(\{V_1, \ldots, V_{n+5}\})$. The three codewords of $\cL_n$ which contain $n+6$ are $4(n+5)(n+6)$, $4(n+6)$, and $34(n+6)$. If we pick a point $p$ in the atom of  $4(n+5)(n+6)$ or $4(n+6)$ with respect to $U_1, \ldots, U_n$, it is now in the atom of $4(n+5)$. If we pick a point in the atom of $34(n+6)$ with respect to $U_1, \ldots, U_{n+6}$, it is in the atom of  $34(n+5)$ with respect to $\{V_1, \ldots, V_{n+6}\}$.

Next, we must check that $V_5 = U_{n+5} \cup U_{n+6}$ is convex.
That is, we must check that for each pair of points $x, y\in U_{n+5}\cup U_{n+6}$, the line segment from $x$ to $y$ is contained in $U_{n+5}\cup U_{n+6}$. 
Without loss of generality, let $x\in U_{n+5}\setminus U_{n+6}$, $y\in U_{n+6}\setminus U_{n+5}$. 
The point $x$ must be contained in the atom of $4(n+5)$ or $4(n+5)(n+4)$. (If $n = 1$,  $4(n+5)(n+4)$ replaces with $24(n+5)(n+4)$.) The point $y$ must be contained in the atom of $4(n+6)$ or $34(n+6)$. In all of these cases, the only feasible path from $x$ to $y$ in $G_{\cL_n}$ includes only codewords containing $n+5$ or $n+6$:
$$ 4(n+5) \lra 4(n+5)(n+6) \lra 4(n+6) \lra 34(n+6).$$ 
Thus the line segment from $x$ to $y$ is contained in $U_{n+6}\cup U_{n+6}$.  See Figure \ref{fig:sun_proof} for an illustration of this argument.


To show that $\cL_n$ is minimally nonconvex, we must show that all codes covered by $\cL_n$ in the poset $\pcode$ are convex. We give a proof of this in Appendix \ref{sec:constructions}, Construction \ref{const:minimal}. 
\end{proof}

Our proof uses ideas similar to the idea of a \emph{rigid structure} in Section 4 of \cite{chan2020nondegenerate}. In particular, our argument that $U_{n+5}\cup U_{n+6}$ must be convex is essentially an open-convex version of a rigid structure, which is a subset of neurons whose union must be convex in any closed-convex realization of a code. 

\begin{figure}[h]\label{fig:sun_proof}
    \centering
    \includegraphics[width = 5 in]{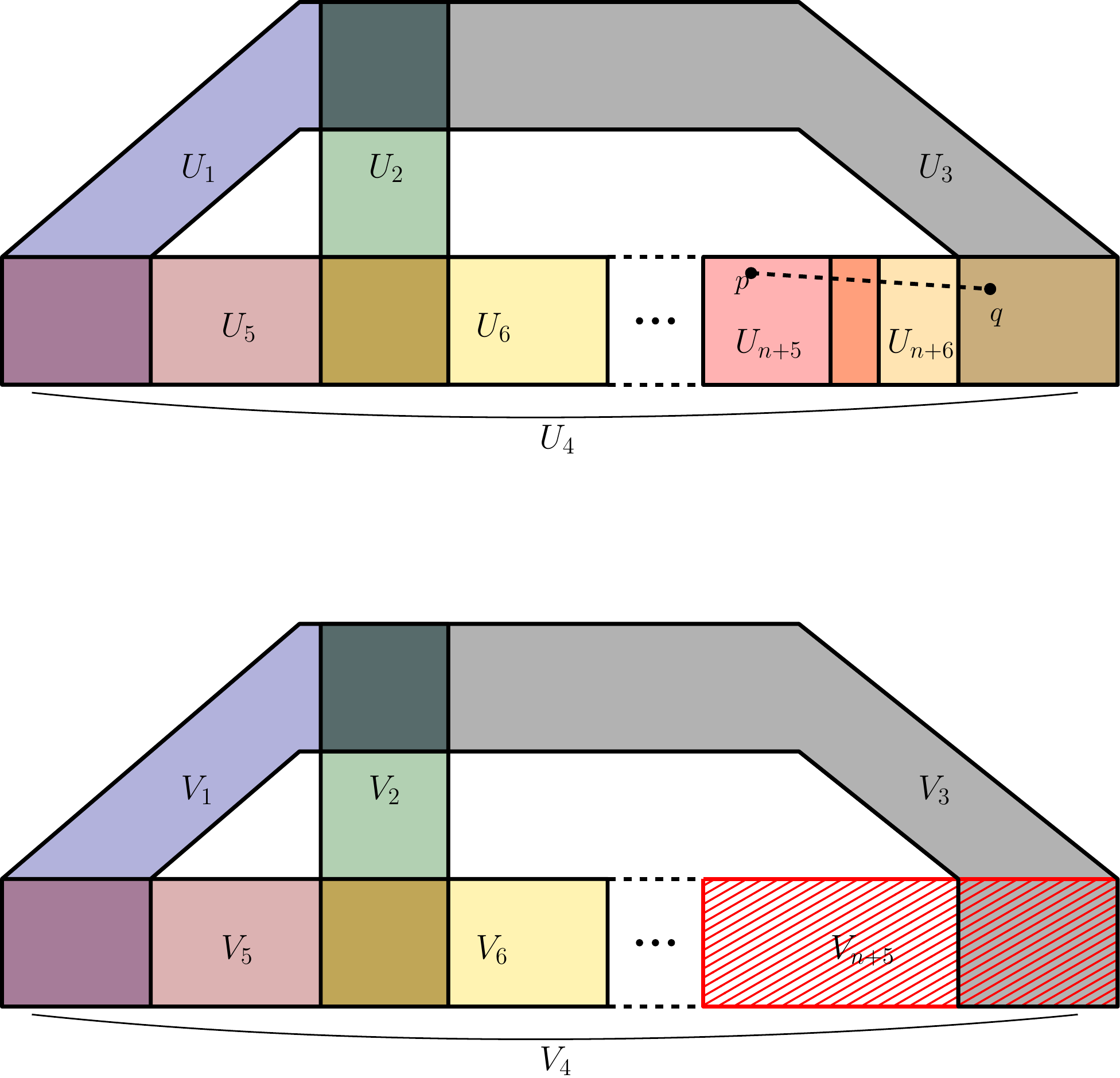}
    \caption{A sketch of the proof of Proposition \ref{prop:stretch_sun}. Since the union of $U_{n+5}$ and $U_{n+6}$ is forced to be convex, we can use a realization of $\cL_{n}$ to construct a realization of $\cL_{n-1}$.}
    \label{fig:sun_proof}
\end{figure}

\subsection{Simple proofs of nonconvexity}\label{sec:braid}
In this section, we give two new examples of good cover codes which are neither open nor closed convex. The proofs that these codes are not convex depend only on order-forcing and elementary geometric arguments. Below, we use lowercase letters for neurons where it would be cumbersome to use only integers.

\begin{center}
\begin{figure}[ht!]
\includegraphics{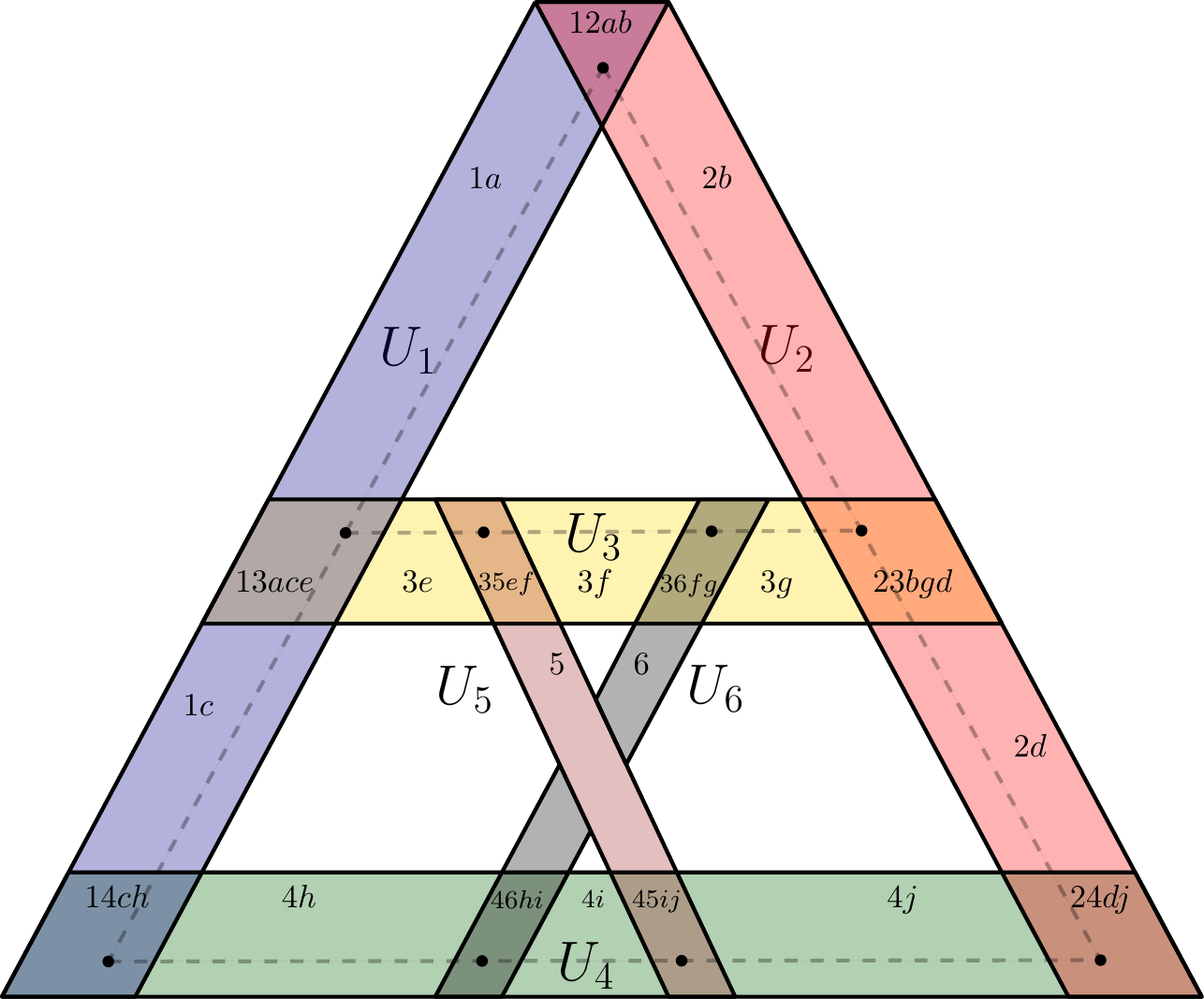}
\caption[something]{A good-cover realization of the non-convex code 
$\cR$ in $\R^3$. The open sets $U_a, U_b, U_c, U_d, U_e, U_f, U_g, U_h, U_i, U_j$ are not shown. Instead, maximal order-forced codewords are noted with vertices, and sets of order-forced vertices are indicated with dashed lines. \label{fig:rowboat}}
\end{figure} 
\end{center}

\begin{prop}\label{prop:rowboat}
The code \begin{align*}
\cR &= \{\mathbf{12ab}, \mathbf{13ace}, \mathbf{14ch}, \mathbf{23bgd},   \mathbf{24dj}, \mathbf{35ef},  \mathbf{36fg},   \mathbf{46hi},  \mathbf{45ij}, \\ &\quad\quad 1a,1c, 2b,2d, 3e,3f,3g,4i,4h, 4j, 5, 6, \emptyset\}
\end{align*} is a good cover code, but is neither open nor closed convex.
\end{prop}

\begin{proof}
We first show that if $\cR$ is convex, then it has a convex realization in the plane. We then show that it does not. Choose points $p_{12}\in A_{12ab}^\cU$, $p_{14}\in A_{14ch}^\cU$, and $p_{24}\in A_{24dj}^\cU$. 
We will use order-forcing to show that each atom of any realization of $\cR$ must have nonempty intersection with $A = \conv(p_{12}, p_{14}, p_{24})$, so that $\{U_i \cap A \mid i\in \{1, \ldots, 6, a, \ldots, h\}\}$ is a convex realization of $\cR$ in $\aff(p_{12}, p_{14}, p_{24})\cong \R^2$. 

First, notice the following order-forced sequences: 
\begin{enumerate} 
\item the only feasible path from $12ab$ to $14ch$ 
is 
$$12ab \leftrightarrow 1a \leftrightarrow 13ace  \leftrightarrow  1c  \leftrightarrow 14ch$$
\item the  only feasible path from $12ab$ to $24dj$ 
is 
$$12ab \lra 2b \lra 23bdg \lra 2d\lra 24dj$$
\item 
the only feasible path from $14ch$ to $24dj$
is 
$$14ch \lra 4h\lra 46hi\lra 45ij \lra 4j \lra 24dj$$
\item 
the only feasible path from $13ace$ to $23bgd$ 
is 
$$13ace \lra 3e \lra 35ef \lra 3f \lra 36fg \lra 3g \lra 23bgd$$
\item the only feasible  path from $35ef$ to $45ij$ 
is 
$$35ef \lra 5 \lra 45ij$$
\item the only feasible path from $36fg$ to $46hi$ 
is 
$$36fg \lra 6\lra 46hi.$$
\end{enumerate}
Now, by Theorem \ref{thm:order-forcing} and order-forcings (1), (2), and (4), the atoms corresponding to codewords $$\{12ab, 1a, 13ace, 1c,  14ch, 2b, 23bdg,  2d, 24dj, 4h, 46hi, 45ij, 4j \}$$ have nonempty intersection with $A$. Thus, we can pick $p_{13}\in A\cap A_{13ace}^\cU$, $p_{23}\in A \cap A_{23bdg}^\cU$, $p_{45} \in A \cap A_{45ij}^\cU$, and $p_{46}\in A\cap A_{46hi}^\cU$. Applying order-forcing (3) to $p_{13}$ and $p_{23}$, we deduce that the atoms corresponding to codewords $$\{3e , 35ef, 3f,  36fg, 3g\}$$ have nonempty intersection with $A$. Thus, we can pick $p_{35} \in A \cap A_{35ef}^\cU$ and $p_{36}\in A\cap A_{36fg}^\cU$. Finally, applying order-forcings (5) and (6), we deduce that the atoms corresponding to codewords $\{5, 6\}$ have nonempty intersection with $A$. This accounts for all codewords of $\cR$. 

Next, we show that $\cR$ cannot have a realization in the plane. Note that by applying an appropriate affine transformation, we can assume that $p_{12}$ is above $p_{14}$ and $p_{24}$, with $p_{14}$ to the left of $p_{24}$, as pictured in Figure \ref{fig:rowboat}. Then by order-forcings (3) and (4), $p_{35}$ must be to the left of $p_{36}$, while $p_{45}$ must be to the right of $p_{46}$. This implies the line segments $\overline{p_{35}p_{45}}$ and $\overline{p_{36}p_{46}}$ must intersect. But if $ p\in \overline{p_{35}p_{45}}\cap \overline{p_{36}p_{46}}$, then $p \in U_5 \cap U_6$. But, since $U_5$ and $U_6$ must be disjoint in any realization of $\cR$, this is not possible. 
\end{proof}

\begin{prop}\label{prop:braid} The code
$$\cT = \{\mathbf{14a}, \mathbf{15ab},  \mathbf{16bg},\mathbf{25c},  \mathbf{24cd}, \mathbf{26dgh},  \mathbf{34e},  \mathbf{35ef},  \mathbf{36fh},$$ $$1a, 1b,2c, 2d,3e,3f, 6g,6h, 4, 5,\emptyset\} $$  is a good cover code, but is neither closed nor open convex.
\end{prop} 

\begin{center}

\begin{figure}[ht!]
\includegraphics{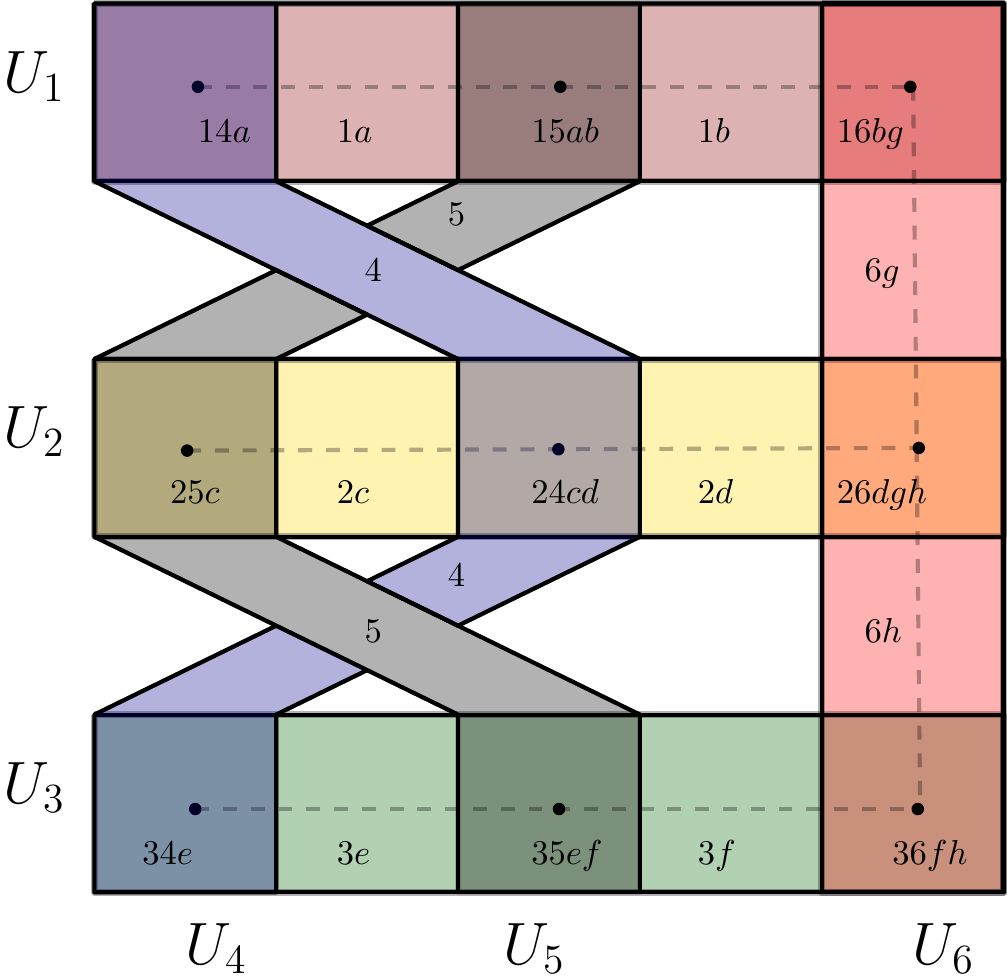}
\caption{A good cover realization of the code 
$$ \cT =  \{\mathbf{14a}, \mathbf{15ab},  \mathbf{16bg},\mathbf{25c},  \mathbf{24cd}, \mathbf{26dgh},  \mathbf{34e},  \mathbf{35ef},  \mathbf{36fh},$$
$$1a, 1b,2c, 2d,3e,3f, 6g,6h, 4, 5,\emptyset\}  
$$
in $\R^3$. The sets $U_1,\ldots, U_6$ are highlighted with various colors while $U_a,\ldots, U_h$ are not highlighted.} 
\label{fig:twist}
\end{figure} 
\end{center}

\begin{proof}
Suppose to the contrary that $\mathcal T$ has a convex realization $\{U_1,\ldots, U_6, U_a,\ldots, U_h\}$. Since the sets $U_4$ and $U_5$ must be disjoint convex sets which are either both open or both closed,  there exists a hyperplane $H$ separating them. In particular, if $U_4$ and $U_5$ are both open, then by the open-set version of the hyperplane separation theorem there is a hyperplane strictly separates them.  That is, $H$ separates $\R^n$ into open half spaces $H^+$ and $H^-$  with $U_4\subseteq H^+$ and $U_5\subseteq H^-$. This also holds if $U_4$ and $U_5$ are both closed. In this case, then without loss of generality, we can choose both sets to be compact. Thus by the compact-set version of the separating hyperplane theorem, there exists a hyperplane $H$ strictly separating them.  We will use order-forcing to exhibit a line segment which crosses  $H$ twice, a contradiction. 

We show that the triples of codewords corresponding to marked points in Figure \ref{fig:twist} are  order-forced. 
More specifically,  we have that:
\begin{enumerate} 
\item the only feasible path from $14a$ to $16bg$ 
is 
$$14a \lra 1a \lra 15ab  \lra 1b  \lra 16bg$$
\item the  only feasible path from $25c$ to $26dgh$ 
is 
$$25c\lra 2c \lra 24cd \lra 2d\lra 26dgh$$
\item  the only feasible path from $34e$ to $36fh$ 
is 
$$34e \lra 3e \lra 35ef \lra 3f \lra 36fh $$

\item the only feasible path from $16bg$ to $36fh$ 
is 
$$16bg \lra 6g\lra 26dgh \lra 6h \lra 36fh.$$
\end{enumerate}

Choose points $p_{14}\in U_{14a}$, $p_{16}\in U_{16bg}$,$p_{25}\in U_{25c}$, $p_{34}\in U_{34e}$, and $p_{36}\in U_{36fh}$. Define line segments $L_1 = \overline{p_{14}p_{16}}$ and $L_3 = \overline{p_{34}p_{36}}$. Notice that by order-forcing (1) we may choose $p_{15}\in L_1\cap U_{15ab}$. Similarly by order-forcing (3) we may choose $p_{35}\in L_3\cap U_{35ef}$.

By ordering forcing (4) we may choose a point $p_{26}\in U_{26dgh}$ on the line segment $\overline{p_{16}p_{36}}$. Lastly, order-forcing (2) allows us to choose a point $p_{24}\in U_{24cd}$ on the line segment $L_2 = \overline{p_{25}p_{26}}$.

Since each of $L_1$ and $L_3$ can only cross $H$ once, the fact that $p_{14}$ and $p_{34}$ are contained in $U_4$, and thus in $H^+$ implies that the points $p_{16}$ and $p_{36}$ are contained in $H^-$. Likewise, the fact that $L_2$ crosses $H$ only once and $p_{25}$ is contained in $U_5$, and thus in $H^+$, implies that the point $p_{26}$ is contained in $H^-$. Thus, the line from $p_{16}$ to $p_{36}$ crosses $H$ twice, a contradiction. 
\end{proof}

Note that both of these codes can be used to generate infinite families of non-convex codes using the same trick we use to produce $\cL_{n}$ from $\cL_0$. 
The codes $\cT$ and $\cR$ do not lie above any previously known non-convex codes in $\pcode$, and in fact are minimally non-convex. This can be checked by exhaustive search of the codes that they cover in $\pcode$, as described in Definition \ref{def:covered}.

\section{Strict Monotonicity of Convexity in All Dimensions}\label{sec:strict_mono}

In this section we turn our attention to a result of \cite{openclosed}, which states that open convexity is a ``monotone" property of codes in the following sense: adding non-maximal codewords to an open convex code preserves its convexity. In fact, \cite{openclosed} proved that adding non-maximal codewords cannot increase the open embedding dimension of a code by more than one. Surprisingly, the same results do not hold for codes that can be realized by closed convex sets \cite{nomonotone}. We now formally state this monotonicity result:

\begin{thm}[Theorem 1.3 in \cite{openclosed}]\label{thm:monotonicity}
Let $\cC\subseteq\cD$ be codes with the same maximal codewords. Then $\odim(\cD)\le \odim(\cC)+1$.
\end{thm}

A natural question arises: under what circumstances is the ``$+1$" term above necessary? That is, when does adding a non-maximal codeword strictly increase the open embedding dimension of a code? The authors in \cite{openclosed} give an example for which $\odim(\cC) = 1$ and $\odim(\cD) = 2$, but do not study the possible gap between $\odim(\cC)$ and $\odim(\cD)$ for codes with larger embedding dimensions.


In the remainder of this section we will show that for any $d\ge 1$ there are codes $\cC\subseteq \cD$ with the same maximal codewords such that $\odim(\cC) = d$ and $\odim(\cD) = d+1$. In other words, we will show that $\odim(\cC)$ may strictly increase when a non-maximal codeword is added to $\cC$, no matter the value of $\odim(\cC)$.

Our approach to proving this result is similar to order-forcing. In particular, we leverage a result about sunflowers of convex open sets (Theorem \ref{thm:sunflower}) to show that a certain atom must appear between other atoms in every realization of a code. This points towards the potential to generalize the order-forcing framework in the future, from examining line segments to examining convex hulls of more than two points (see Question \ref{q:higherdimensional}). This is particularly promising because our results in this section allow us to determine the exact open embedding dimension of certain codes, which is more refined information than whether or not they are convex.

Our primary tool will be the following family of codes.  Below, we make use of neurons decorated by an overline to simplify our notation. In particular, when $\sigma\subseteq [n]$, we let $\overline\sigma = \{\overline i \mid i \in\sigma\}$,  and the neurons in $\overline\sigma$ are distinct from those in $\sigma$.

\begin{defn}\label{def:Pd}
Let $d\ge 1$. The $d$-th \emph{prism code}, denoted $\cP_d$, is the code on neurons $\{1,2,\ldots, d+1\}\cup\{\overline{1},\overline{2},\ldots, \overline{d+2}\}$ which has the following codewords:\begin{itemize}
    \item[(i)] All subsets of $\{1,2,\ldots, d+1,\overline{i}\}\setminus \{i\}$ for each $i$ between 1 and $d+1$,
    \item[(ii)] $\{\overline{1},\overline{2},\ldots, \overline{d+2}\}$.
\end{itemize}
\end{defn}

Observe that $\cP_d$ is intersection complete and has $d+2$ maximal codewords. We start by characterizing the open embedding dimension of $\cP_d$.

\begin{prop}\label{prop:Pddim}
$\odim(\cP_d) = d$. 
\end{prop}
\begin{proof}
Note that every $d$-subset of $[d+1]$ appears in some codeword of $\cP_d$ but $[d+1]$ does not appear in any codeword. This implies that the receptive fields of the neurons $1,2,\ldots, d+1$ in any realization of $\cP_d$ will have a nerve that is the boundary of a $d$-simplex. This can only occur in dimension $d$ or higher (see \cite[Section 1.2]{tancer} or \cite{leray} for further details), and so $\odim(\cP_d)\ge d$. To prove that $\odim(\cP_d) \le d$, we must exhibit a realization of $\cP_d$ in $\R^d$. This is done in in Appendix \ref{sec:constructions} with Construction \ref{const:Pd}. When $d=2$, this construction is shown in Example \ref{ex:prism} below.
\end{proof}
\begin{ex}\label{ex:prism}
Figure \ref{fig:P2} shows a realization of \[\cP_2 = \{\mathbf{12\overline 3}, \mathbf{1\overline 2 3}, \mathbf{\overline 1 23}, \mathbf{\overline{1234}}, 12, 13, 23, 1\overline 2, 1\overline 3,  2\overline 1, 2\overline 3, 3 \overline 1 , 3 \overline 2 , 1, 2, 3, \overline 1, \overline 2, \overline 3, \emptyset\}\] in $\R^2$ as given by the proof of Proposition \ref{prop:Pddim}. The ``shaved off" regions arise at the rounded corners, realizing the codewords $12$, $13$, and $23$. Observe that we could not modify this realization to add the codeword $\{\overline 4\}$ since any extension of $V_4$ outside the central triangle would overlap $V_1$, $V_2$, or $V_3$. We will formalize this observation in Theorem \ref{thm:strictmonotonicity}.

\begin{figure}[h]\includegraphics[width=30em]{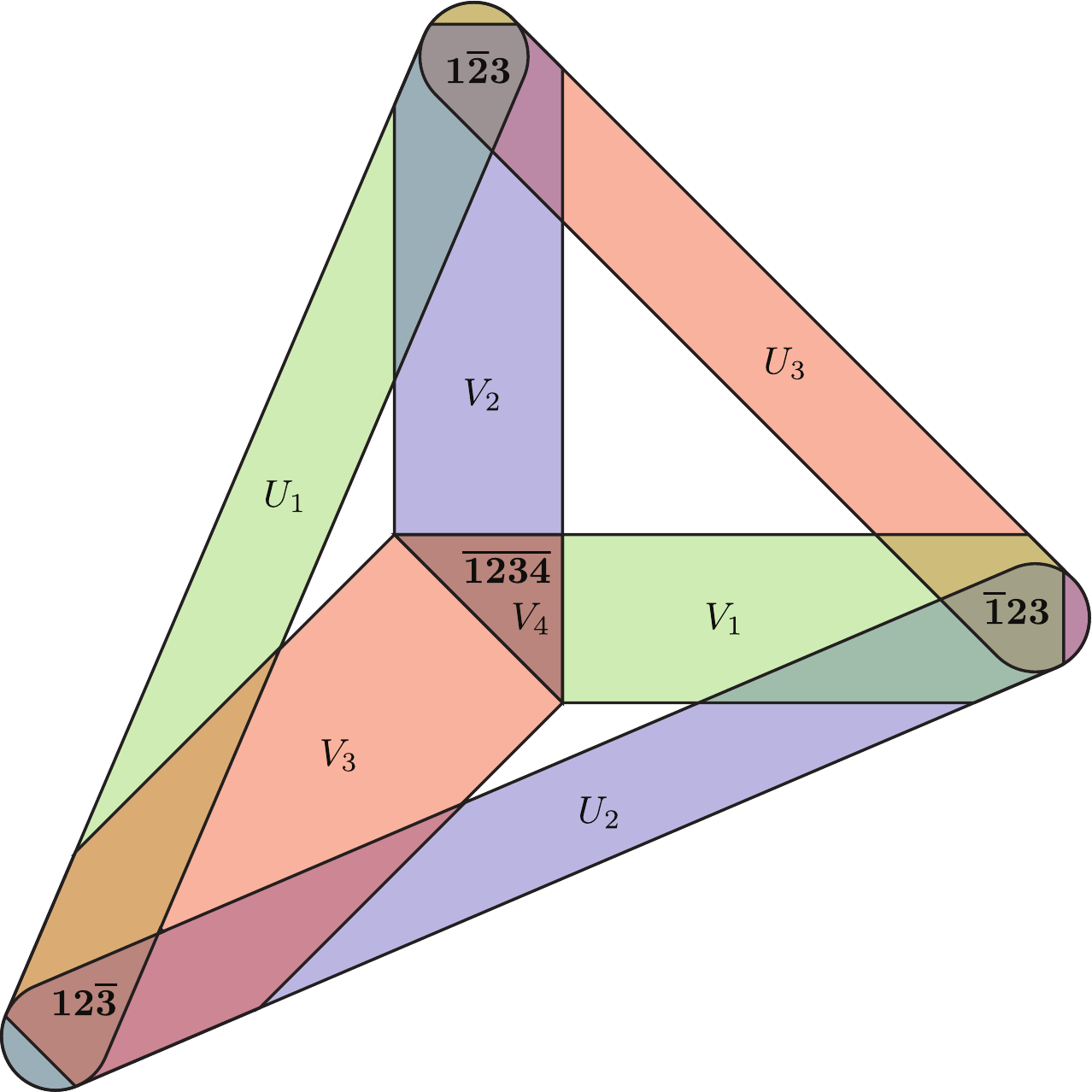}
\caption{A realization of $\cP_2$ in $\R^2$. }\label{fig:P2}
\end{figure}
\end{ex}

To prove our monotonicity result, we apply a recent geometric result on sunflowers of convex open sets, introduced below.

\begin{defn}
Let $\cU = \{U_1,\ldots, U_n\}$ be a collection of convex open sets in $\R^d$. We say that $\cU$ is a \emph{sunflower} if $U_i\cap U_j = \bigcap_{k\in[n]} U_k\neq \emptyset$ for all $i\neq j$. Equivalently, $\cU$ is a sunflower if $\code(\cU)$ contains the codeword $[n]$, and all other codewords contain no more than one neuron. When $\cU$ is a sunflower, the $U_i$ are called \emph{petals} and $\bigcap_{k\in[n]} U_i$ is called the \emph{center} of the sunflower.
\end{defn}

\begin{thm}[Corollary 2.1 in \cite{sunflowers}] \label{thm:sunflower}
Let $\cU = \{U_1,\ldots, U_{d+1}\}$ be a sunflower in $\R^d$, and for $i\in[d+1]$ choose $p_i\in U_i$. Then $\conv\{p_1,\ldots, p_{d+1}\}$ contains a point in the center of $\cU$.
\end{thm}

Roughly, Theorem \ref{thm:sunflower} says that the petals of a sunflower $\cU$ are rigid in the sense that not too many can point in the same direction. 
This result is relevant to our family of codes $\cP_d$ because the receptive fields of neurons $\overline 1,\ldots, \overline{d+2}$ in a realization of $\cP_d$ will always form a sunflower. 
\begin{thm}\label{thm:strictmonotonicity}
Let $d\ge 1$. Then $\odim(\cP_d\cup \{\{\overline{d+2}\}\}) = d+1$. That is, adding a non-maximal codeword to $\cP_d$ may increase its open embedding dimension.
\end{thm}
\begin{proof}
Let $\cC = \cP_d\cup \{\{\overline{d+2}\}\}$. By Theorem \ref{thm:monotonicity} and Proposition \ref{prop:Pddim} we know that $\odim(\cC)\le d+1$. It remains to show that there is no realization of $\cC$ in $\R^d$.

Suppose for contradiction that there exists a realization $\{U_1,\ldots, U_{d+1}, V_1,\ldots, V_{d+2}\}$ of $\cC$ in $\R^d$, where $U_i$ corresponds to the neuron $i$ and $V_i$ corresponds to the neuron $\overline{i}$. Since the only codeword in $\cC$ containing $\overline i$ and $\overline j$ for $i\neq j$ is $\{\overline 1,\ldots, \overline{d+2}\}$, the sets $V_1,\ldots, V_{d+2}$ must form a sunflower in $\R^d$. Call its center $V$, and for each $i$ between $1$ and $d+1$ choose a point $p_i$ in the intersection $V_i\cap \bigcap_{j\neq i} U_j$. Since the region $V_i\cap \bigcap_{j\neq i} U_j$ is open, we may assume that the $p_i$ are in general position, so that their convex hull is a $d$-simplex.

Theorem \ref{thm:sunflower} tells us that $\conv\{p_1,\ldots, p_{d+1}\}$ contains a point in $V$. In fact, we claim that this convex hull contains the entirety of $V$. If not, then $V$ would have to cross one of the facets of the simplex $\conv\{p_1,\ldots, p_{d+1}\}$ (since we are working in $\R^d$). By choice of $p_i$ each of these facets is contained in some $U_j$. Since $V$ is disjoint from all $U_j$, it cannot cross any of these facets. Thus $V$ is contained in $\conv\{p_1,\ldots, p_{d+1}\}$.

 Since $\{d+2\}$ is a codeword in $\cC$, we may choose a point $p\in V_{d+2}\setminus V$, and examine a generic line segment $L$ from $p$ to a point in $V$. Let $H$ be a hyperplane supporting $V$ at the point where $L$ crosses the boundary of $V$. Observe that $H$ separates $p$ from $V$, and moreover no $p_i$ may lie on the same side of $H$ as $p$, otherwise $V_i$ and $V_{d+2}$ would meet outside of $V$, violating the definition of a sunflower. Thus $H$ separates $p$ from all $p_i$, and $p$ lies outside $\conv\{p_1,\ldots, p_{d+1}\}$.

Since $V$ is contained in $\conv\{p_1,\ldots, p_{d+1}\}$, the line $L$ crosses the boundary of this simplex. Since $L$ is contained in $V_{d+2}$ this implies that $V_{d+2}$ intersects some $U_i$. But the only codewords of $\cC$ containing $\overline{d+2}$ are $\{\overline{d+2}\}$ and $\{\overline{1},\overline{2},\ldots, \overline{d+2}\}$, so this is a contradiction. Thus $\cC$ is not convex in $\R^d$ and the result follows.
\end{proof}

\section{Conclusion and Open Questions}\label{sec:conclusion}

Past work constructing non-convex codes has used notions that are similar to, but distinct from, order-forcing. For example, sunflower theorems such as \cite[Theorem 1.1]{sunflowers} and \cite[Theorem 1.11]{embeddingphenomena} were used to show that the convex hull of points sampled from certain atoms in a convex realization must intersect another atom. Likewise, \cite{CUR} used collapses of simplicial complexes to prove that in certain codes the convex hull of appropriately chosen points must intersect certain atoms.

Order-forcing brings a new perspective to this general approach: not only must certain atoms appear, but they must appear in a certain arrangement (i.e. in a particular order along a line segment). The order of points on a line may be generalized to higher dimensions by examining the ``order type" of a point configuration \cite{ordertype}. We thus ask the following.

\begin{question}\label{q:higherdimensional}
Does there exist a general result connecting the combinatorial structure of a code $\cC$ to the order type of points chosen from certain atoms in any convex realization of $\cC$? Can such a result be formulated so that the connections between convex codes and sunflower theorems \cite{sunflowers,embeddingphenomena}, convex union representable complexes \cite{CUR}, or oriented matroids \cite{matroids} are special cases?
\end{question}

A cleanly formulated answer to Question \ref{q:higherdimensional} would allow us to create fundamentally new families of non-convex codes.

To connect the combinatorics of order-forcing with the geometry of convex realizations, we examined straight line segments between different atoms. One could try to replace convex realizations by good cover realizations, and straight lines by continuous paths, which leads to the following question.

\begin{question}
If $\cC$ is a good cover code, are there feasible paths between all pairs of codewords in $\cC$?
\end{question}

Our examples have used order-forcing to prove that codes are not convex. However, even if a code is convex, one might hope to use order-forcing to bound its open or closed embedding dimension. 

\begin{question}
Can one use order-forcing to provide new lower bounds on the open or closed embedding dimension of codes?
\end{question}

Morphisms and minors of codes have played a role in characterizing ``minimal" obstructions to convexity, contextualizing results, and systematizing the study of convex codes \cite{morphisms,embeddingphenomena}. It would be interesting to phrase our results in this framework.

\begin{question}
How does order-forcing interact with code morphisms and minors? If $f:\cC\to \cD$ is a morphism, and $\sigma_1,\sigma_2,\ldots,\sigma_k$ is an order-forced sequence in $\cC$, under what conditions is $f(\sigma_1),f(\sigma_2),\ldots, f(\sigma_k)$ order-forced in $\cD$? Similarly, if $f$ is surjective and $\tau_1,\tau_2,\ldots, \tau_k$ is order-forced in $\cD$, when can we find $\sigma_1,\ldots, \sigma_k$ order-forced in $\cC$ with $f(\sigma_i) = \tau_i$ (i.e., when can we ``pull back" an order-forced sequence)?
\end{question}

Work in \cite{matroids} used minors of codes to tie the study of convex codes to the study of oriented matroids, in particular showing that non-convex codes come in two types: those that are minors of non-representable oriented matroid codes, and those that are not minors of any oriented matroid code. Concretely, it would be useful to understand which of these classes our codes $\cT$ and $\cR$ fall into.

\begin{question}
Are the codes $\cT$ and $\cR$ from Section \ref{sec:of_examples} minors of oriented matroid codes?  
\end{question}

\section{Acknowledgements}
The authors are grateful to Anne Shiu for detailed comments on an earlier draft. CL and NY are grateful to Isabella Novik for inviting them on a research visit to the University of Washington, where this project began. CL was supported by the NSF (DGE-1255832). RAJ was supported by the NSF (DGE-1761124).
\appendix
\section{Constructions of Various Realizations}\label{sec:constructions}
\begin{const}
\label{const:minimal}
In order to check that $\cL_n$ is minimally non-convex for all $n$, we must show that all codes covered by $\cL_n$ in $\pcode$ are convex. For this, we need the following characterization, from \cite{sunflowers}, of the covering relations in $\pcode$. 

\begin{defn}[Definition 3.9 of \cite{sunflowers}]\label{def:covered}
Let $\cC\subseteq 2^{[n]}$ be a code, let $i\in[n]$, and let $\sigma = [n]\setminus \{i\}$. Consider the morphism $f_i:\cC\to 2^{\sigma\cup\overline\sigma}$ defined by \[
f(c) = \begin{cases} c\cap\sigma & i\notin c,\\
c\cap\sigma \cup (\overline{c\cap\sigma}) & i\in c.\end{cases}
\]
The \emph{$i$-th covered code} of $\cC$ is the image of $\cC$ under $f_i$, and is denoted $\cC^{(i)}$.
\end{defn}

Importantly, if a code $\cD$ is covered by $\cC$ in $\pcode$, then $\cD$ must be one of the covered codes described above. Thus to prove that a non-convex code $\cC$ is minimally non-convex, it suffices to prove that all of its covered codes are convex. 

A useful geometric interpretation of covered codes is as follows. Suppose that $\cU = \{U_1,\ldots, U_n\}$ is a (possibly not convex) realization of $\cC$. Then we may obtain a realization of $\cC^{(i)}$ by deleting $U_i$ from $\cU$, and adding sets $U_{\overline j} = U_i\cap U_j$ for all $j\neq i$.

In some cases, there may be distinct neurons $j, k$ such that $U_{\bar j }= U_{\bar k}$. In this cases, one of the neurons $\bar j, \bar k$ is  redundant, and we can remove it from the code without discarding geometric information. More generally, a neuron $j$ is \emph{redundant}  to a set $\sigma \subseteq [n]\setminus \{j\}$ if $\tk_\cC(j) = \tk_\cC(\sigma)$, and a neuron is \emph{trivial} if it does not appear in any codeword \cite{morphisms}. A code is \emph{reduced} if it does not have any trivial or redundant neurons. Theorem 1.4 of \cite{morphisms} states that a code is always isomorphic to a reduced code. Thus, convexity of the reduced code is equivalent to convexity of the original code. Thus, we can ``clean up" $\cC^{(i)}$ by removing all trivial or redundant neurons. In what follows, we give realizations for reduced versions of all codes mentioned. 

Thus, to show that $\cL_n$ is minimal for all $n$, we need to construct realizations for each covered code $\cL_n^{(i)}$. In Figure \ref{fig:const123}, we construct realizations of $\cL_n^{(1)}$, $\cL_n^{(2)}$, and $\cL_n^{(3)}$ in $\R^2$. In Figure \ref{fig:const4}, we construct a realization of $\cL_n^{(4)}$ in $\R^3$. Finally, in Figure \ref{fig:const56}, we construct a convex realization of $\cL_n^{(7)}$ in $\R^3.$ An analogous process can be used to construct convex realizations of $\cL_n^{(8)}, \ldots,  \cL_n^{(n+6)}.$

\begin{figure}
    \centering
    \includegraphics[width = 4 in]{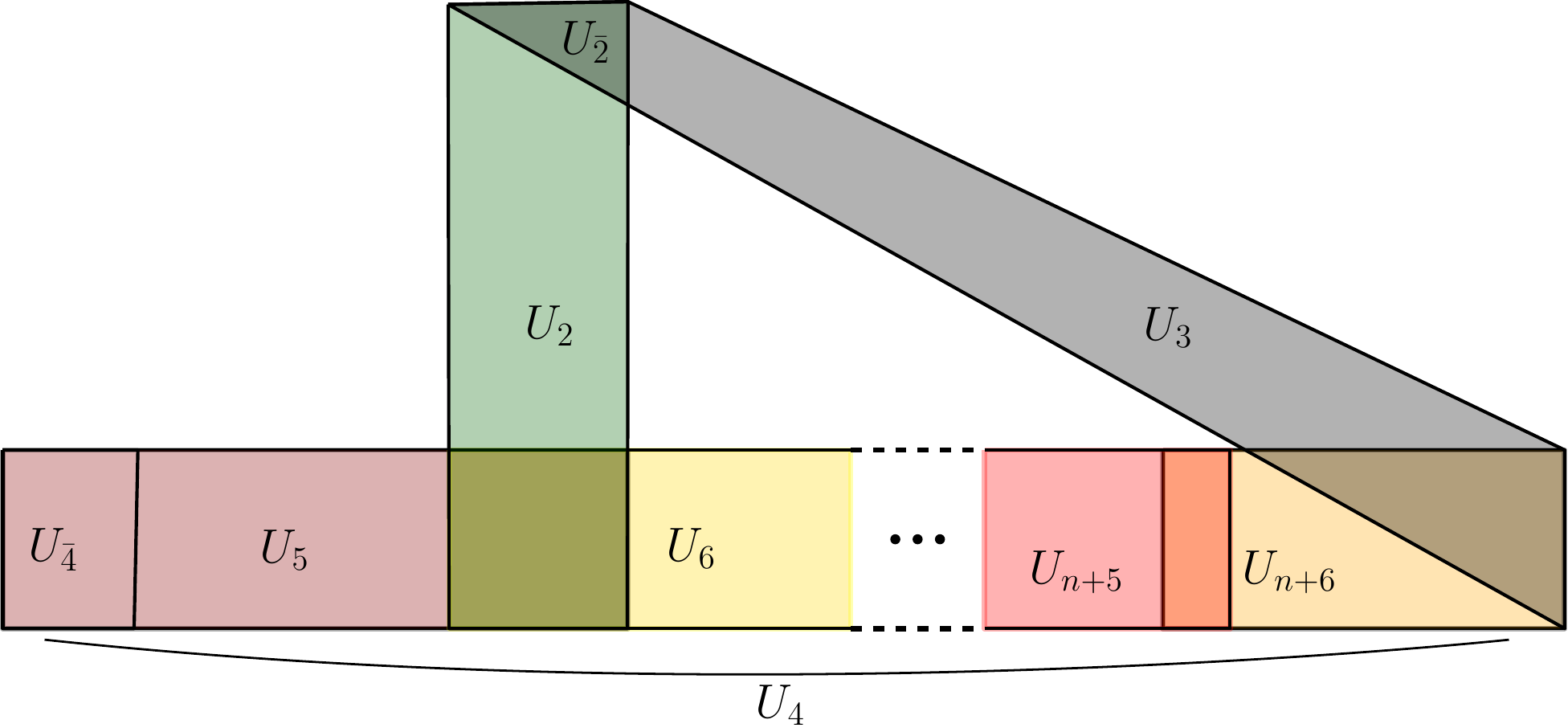}
    \includegraphics[width = 4 in]{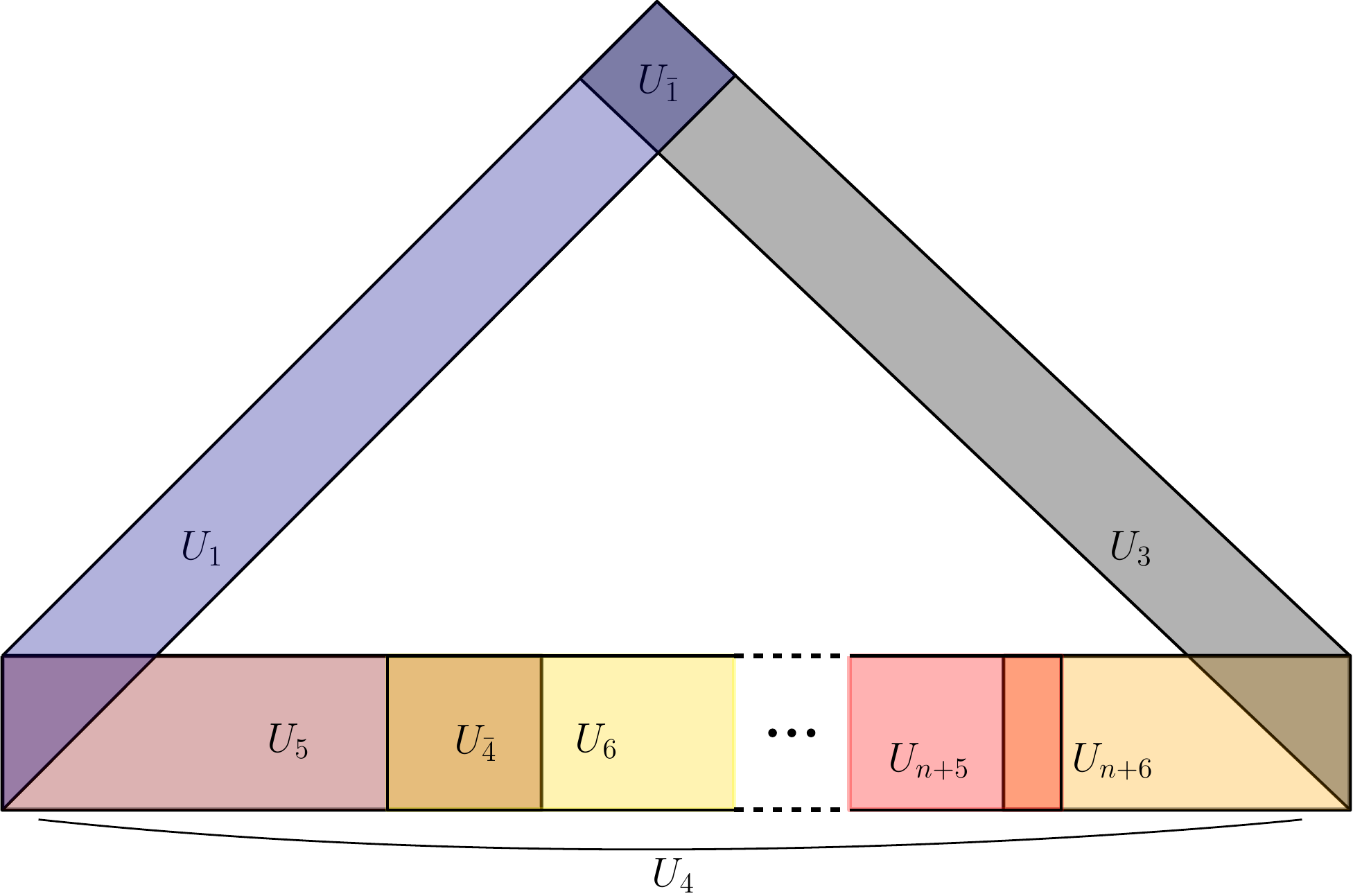}
    \includegraphics[width = 4 in]{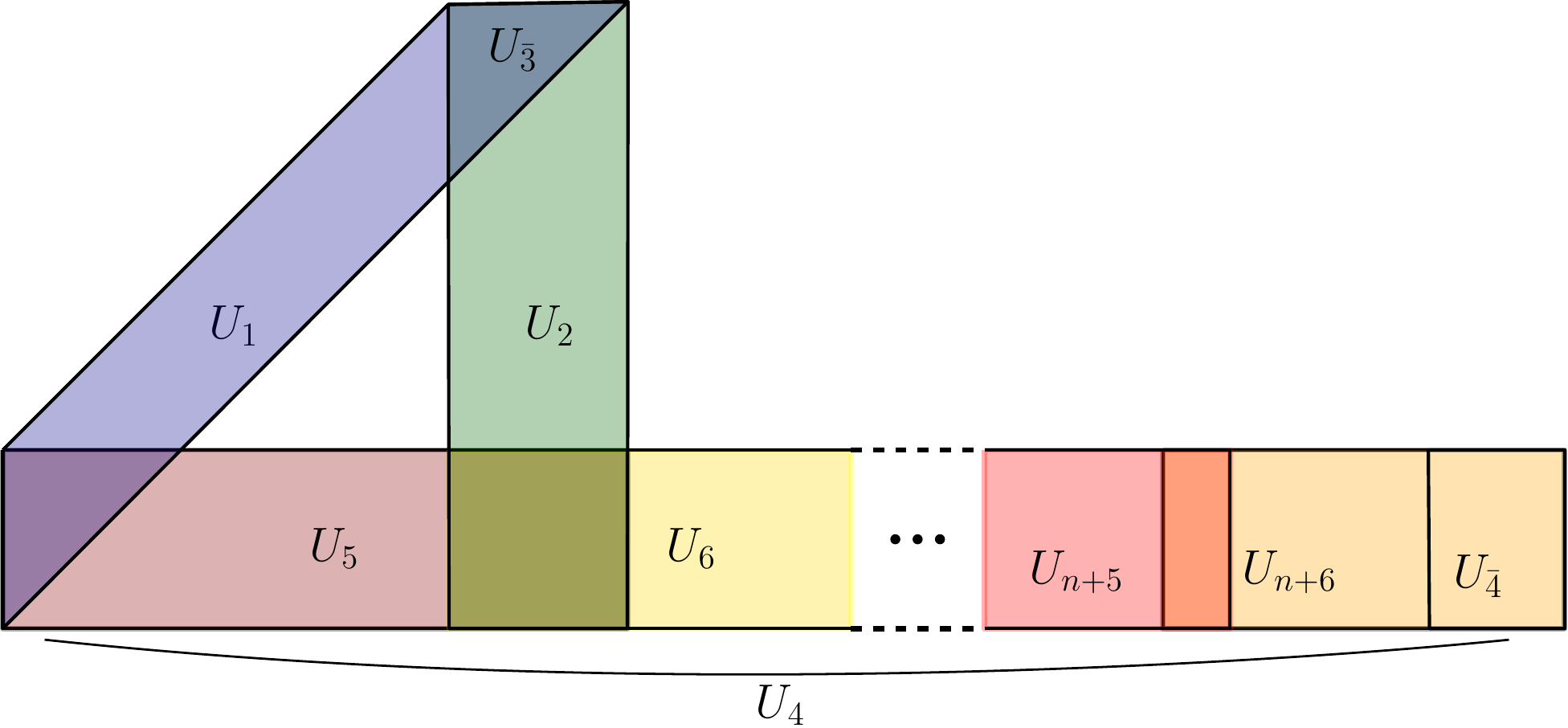}
    \caption{Convex realizations in $\R^2$ of the codes $\cL_{n}^{(1)}$, $\cL_{n}^{(2)}$, $\cL_{n}^{(3)}$.}
    \label{fig:const123}
\end{figure}

\begin{figure}
    \centering
    \includegraphics{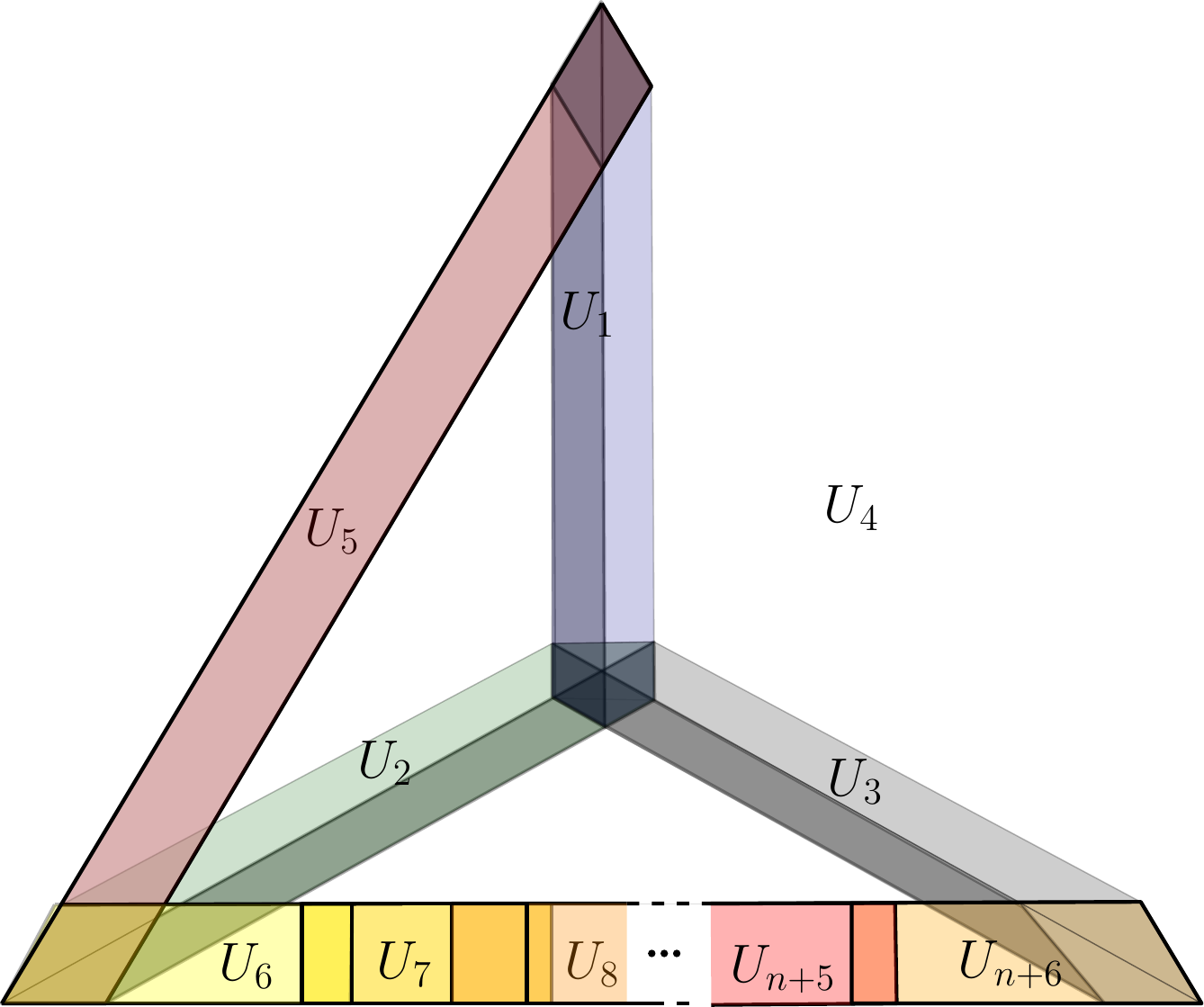}
    \caption{A convex realization in $\R^3$ of the code $\cL_n^{(7)}$. An analogous construction can be used to construct convex realizations for $\cL_n^{(5)}$, 
    $\cL_n^{(6)}$, and $\cL_{n}^{(8)}, \ldots, \cL_{n}^{(n+6)}$ in $\R^3$. }
    \label{fig:const4}
\end{figure}

\begin{figure}
    \centering
    \includegraphics{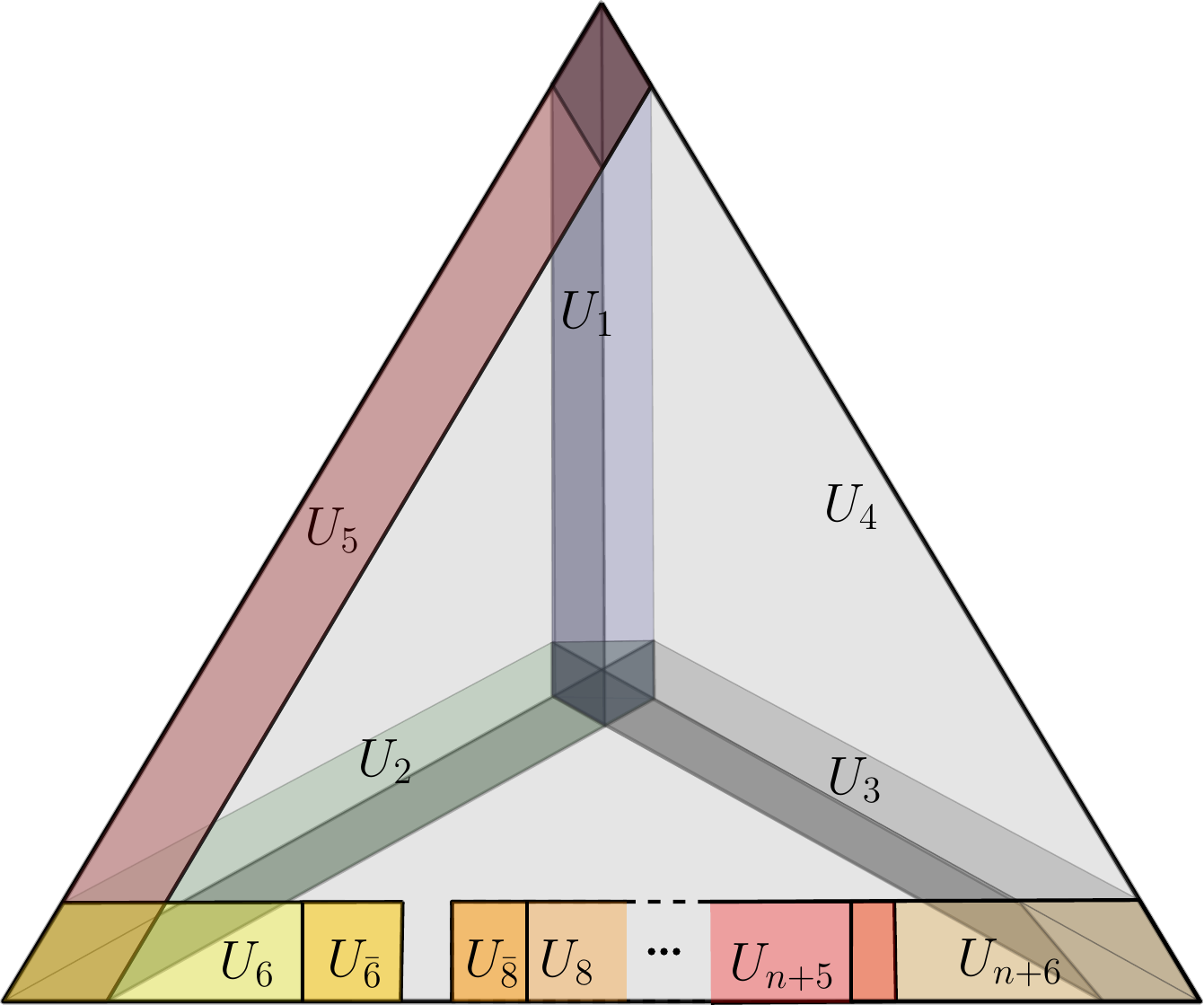}
    \caption{A convex realization in $\R^3$ of the code $\cL_n^{(7)}$. An analogous construction can be used to construct convex realizations for $\cL_n^{(5)}$, 
    $\cL_n^{(6)}$, and $\cL_{n}^{(8)}, \ldots, \cL_{n}^{(n+6)}$ in $\R^3$. }
    \label{fig:const56}
\end{figure}

\end{const}

\begin{const}\label{const:Pd}
The code $\cP_d$ of Definition \ref{def:Pd} has an open convex realization in $\R^d$.
\end{const}
\begin{proof}
 First choose points $p_1,\ldots, p_d$ where $p_i = e_i$ in $\R^d$. Also choose $p_{d+1} = -\mathbf{1}$, the vector whose entries are all $-1$. For $i\in [d+1]$ define $F_i$ to be the facet $\conv\{p_j\mid j\neq i\}$ of the $d$-simplex $\conv\{p_1,\ldots, p_{d+1}\}$, and let $U_i$ to be the Minkowski sum of $F_i$ with a small ball of radius $\varepsilon$. Choose a small $d$-simplex with center of mass at the origin and facet normal vectors equal to the various $p_i$, and let $V$ denote its interior. For $i\in[d+1]$ define $V_i$ to be Minkowski sum of $V$ with a ray in the direction of $p_i$. Lastly, define $V_{d+2}$ to be equal to $V$. 

Observe that we may choose $V$ small enough that its closure is contained in the interior of $\conv\{p_1,\ldots, p_{d+1}\}$. We may then choose $\varepsilon$ small enough that the various $U_i$ do not intersect $V$. Let $\cD$ denote the code arising from this realization. We claim that $\cD$ has the same maximal codewords as $\cP_d$, and that $\cD\subseteq \cP_d$.

Let us first determine the maximal codewords that arise in $\cD$. One maximal codeword is $\{\overline 1,\ldots, \overline{d+2}\}$, which arises only inside $V$. The codeword $\{1,2,\ldots, d+1,\overline{i}\}\setminus \{i\}$ arises in a small neighborhood of the point $p_i$, and it is maximal since this neighborhood can be separated from $U_i$ and all $V_j$ with $j\neq i$ by a hyperplane with normal vector equal to $p_i$. This shows that the maximal codewords of $\cP_d$ arise as maximal codewords in $\cD$.

We must argue that no other maximal codewords arise. Clearly the only maximal codeword containing $\overline{d+2}$ is $\{\overline 1,\ldots, \overline{d+2}\}$ since $V = V_{d+2}$ is disjoint from all $U_i$. The other possibilities are a maximal codeword that contains $[d+1]$ or a maximal codeword that contains $\{i,\overline i\}$ for some $i\in [d+1]$. The former is impossible since we have chosen $\varepsilon$ small enough that various $U_i$ do not intersect $V$ and thus do not all meet at a single point. The latter is impossible because $V_i$ and $U_i$ are separated by a hyperplane parallel to $F_i$. Thus the maximal codewords arising in $\cD$ are exactly those in $\cP_d$.

We next show that the non-maximal codewords in $\cD$ are codewords in $\cP_d$. First let us consider the codewords of $\cD$ that do not contain any $i\in[d+1]$. By construction the various $V_i$ only overlap inside $V$, and so the only codewords of this type in $\cD$ are the singleton codewords $\{\overline{i}\}$ for $i\in[d+1]$, which arise near the face of $V$ with normal vector $p_i$ (and $V_{d+2} = V$, so $\{\overline{d+2}\}$ does not arise as a singleton codeword). Any other codeword of $\cD$ contains a neuron $i\in[d+1]$ and is thus contained in some maximal codeword $\{1,2,\ldots, d+1,\overline i\}\setminus \{i\}$. But $\cP_d$ contains all subsets of $\{1,2,\ldots, d+1,\overline i\}\setminus \{i\}$ and so we conclude that $\cD\subseteq \cP_d$.

We now modify our realization $\{U_1,\ldots, U_{d+1}, V_1,\ldots, V_{d+2}\}$ to obtain a realization of $\cP_d$ by applying techniques of \cite[Section 5.4]{openclosed}. Since we are working in $\R^d$ we may choose an open ball $B$ whose boundary contains every $p_i$. Let us replace every set in our realization by its intersection with $B$. The resulting code is still contained in $\cP_d$, and since $B$ contains the interior of $\conv\{p_1,\ldots, p_{d+1}\}$ it still has the same maximal codewords as $\cP_d$. Moreover, the regions in our realization corresponding to the maximal codewords  $\{1,2,\ldots, d+1,\overline i\}\setminus \{i\}$ now have closures which intersect the boundary $\partial B$ of $B$ in a relatively open subset of $\partial B$. In the proof technique of \cite[Lemma 5.7]{openclosed} implies that we may repeatedly ``shave off" pieces of our sets near this region to add the desired non-maximal codewords contained in $\{1,2,\ldots, d+1,\overline i\}\setminus \{i\}$, obtaining a realization of $\cP_d$ in $\R^d$.
\end{proof}

\newpage
\bibliography{neuralcodereferences.bib}

\end{document}